\documentclass[a4paper, 12pt, reqno]{amsart}

\usepackage{amsmath, amsfonts, amsthm, amssymb}
\usepackage[utf8]{inputenc} 
\usepackage[T1]{fontenc}   
\usepackage{graphicx}

\usepackage[table]{xcolor}
\usepackage{stmaryrd}
\usepackage{bm}
\usepackage{bbm}
\usepackage{tikz-cd}
\usepackage[shortlabels]{enumitem}
\setlist[1]{wide}
\setlist[2]{leftmargin=15mm}
\setlist[enumerate]{label=\rm{(\roman*)}}
\setlist[enumerate,2]{label=\rm({\it\roman*}), }
\setlist[itemize]{label=\raisebox{0.25ex}{\tiny$\bullet$}}

\usepackage[backref, colorlinks, linktocpage, citecolor = blue, linkcolor = blue]{hyperref}

\usepackage{soul}
\usepackage{mathtools}

\usepackage{lmodern}
\usepackage[noadjust]{cite}
\usepackage{framed}
\usepackage{stackrel}
\usepackage{array}
\usepackage{xifthen}

\usepackage[english]{babel}
\usepackage{csquotes}

\DeclareQuoteStyle[american]{english}
{\itshape\textquotedblleft}
[\textquotedblleft]
{\textquotedblright}
[0.05em]
{\textquoteleft}
{\textquoteright}

\usepackage[margin=2cm]{geometry}
\usepackage{booktabs}
\usepackage{longtable}

\usepackage[all, cmtip]{xy}

\makeatletter
\renewcommand\subsection{\@startsection{subsection}{1}%
  \z@{.5\linespacing\@plus.7\linespacing}{-.5em}%
  {\normalfont}}
\def\@seccntformat#1{%
  \protect\textup{\protect\@secnumfont
    \ifnum\pdfstrcmp{subsection}{#1}=0 \bfseries\fi% subsection # in \bfseries
    \csname the#1\endcsname
    \protect\@secnumpunct
  }%
}  
\@namedef{subjclassname@2020}{\textup{2020} Mathematics Subject Classification}
\makeatother

\numberwithin{equation}{section}

\newtheorem{theorem}{Theorem}[section]
\newtheorem{proposition}[theorem]{Proposition}
\newtheorem{lemma}[theorem]{Lemma}
\newtheorem{corollary}[theorem]{Corollary}
\newtheorem{question}[theorem]{Question}
\theoremstyle{definition}

\newtheorem{definition}[theorem]{Definition}

\newtheoremstyle{customNumber}
     {}          % Space above (empty = default)
     {}          % Space below
     {\itshape}  % Body font
     {}          % Indent amount (empty = no indent)
     {\bfseries} % Thm head font
     {.}         % Punctuation after thm head
     { }         % Space after thm head (\newline = linebreak)
     {\thmname{#1}\thmnumber{ #2}\thmnote{ #3}}
\theoremstyle{customNumber}

\renewcommand{\phi}{\varphi}

\DeclareMathOperator{\Z}{\mathbb{Z}}

\newcommand{\easy}{\operatorname}
\newcommand{\RomanNumeralCaps}[1]
{\MakeUppercase{\romannumeral #1}}

\title[Real forms on rational surfaces]
{Real forms on rational surfaces}

\author{Anna Bot}
\address{}

\subjclass[2020]{14J50, 14J26, 14P05}
\keywords{Rational surfaces, real forms, real structures}

\begin{document}

\begin{abstract}
	For any positive integer $r$, we construct a smooth complex projective rational surface which has at least $r$ real forms not isomorphic over $\mathbb{R}$.
\end{abstract}

\thanks{}
\maketitle

\tableofcontents

\section{Introduction}

% MSC classes: 
% 14J50 (14J26, 14P05)

\par A \textit{real form} of a complex algebraic variety $X$ is a real algebraic variety whose complexification is isomorphic to $X$. For many families of complex projective algebraic varieties it is known that the number of isomorphism classes over $\mathbb{R}$ of real forms is finite, as for example in the cases of abelian varieties \cite{MR678890}, \cite[{Appendix D}]{MR1795406} (based on \cite{MR181643} for abelian varieties as algebraic groups), algebraic surfaces of Kodaira dimension greater than or equal to one \cite[{Appendix D}]{MR1795406}, minimal algebraic surfaces \cite[{Appendix D}]{MR1795406}, Del Pezzo surfaces \cite{MR1954070} and compact hyperk\"ahler manifolds \cite{MR4023392}. Different authors have raised this finiteness question for nonminimal rational surfaces: \cite[pages 232-233]{MR1795406}, \cite[Problem, page 1128]{MR3473660}, \cite[page 943]{MR3934593} and \cite[Question 1.5]{2020arXiv200204737D}. More precisely, at the time this paper appeared, the following question was still open: 
\begin{question}[{\cite[Problem, page 1128]{MR3473660}, \cite[Question 1.5]{2020arXiv200204737D}}] \label{question}
	Is there a smooth complex projective rational surface with infinitely many mutually nonisomorphic real forms?
\end{question}
\par The first example of a complex projective variety with infinitely many nonisomorphic real forms was obtained by Lesieutre \cite[Theorem 2]{MR3773792}, for varieties of dimension $6$. It was then generalised in \cite[Theorem 1.1]{MR3934593} by Dinh and Oguiso to any dimension $d \geq 2$ for varieties of Kodaira dimension $d-2$, and in \cite{2020arXiv200204737D} by Dinh, Oguiso and Yu to any dimension greater or equal to three for rational varieties. See also Dubouloz, Freudenburg and Moser-Jauslin \cite{dubouloz2020smooth} for examples of rational affine varieties of dimension greater or equal to four. But as expressed by Kharlamov in \cite{MR1936747}, \enquote{[...] what concerns this finiteness problem, the case of non minimal rational surfaces looks to be the most complicated and enigmatic,} or as Cattaneo and Fu \cite{MR4023392} describe, \enquote{[the] remaining biggest challenge for surfaces seems to be the case of rational surfaces.}

\par In fact, in \cite{{MR3773792}, {MR3934593}, {2020arXiv200204737D}}, all constructions of projective varieties with infinitely many nonisomorphic real forms come from projective varieties $X$ such that $\easy{Aut}(X)/\easy{Aut}^0(X)$ is not finitely generated. Whether $\easy{Aut}(X)/\easy{Aut}^0(X)$ is not finitely generated for projective rational surfaces is still open. As an indication of where to look for potentially infinitely many isomorphism classes of real forms in this situation, Benzerga \cite{MR3473660} proved that if such a complex projective rational surface exists, then it will be the blow-up of $\mathbb{P}^2_{\mathbb{C}}$ in at least ten points and have at least one automorphism of positive entropy. Also, due to \cite{MR4060185}, the automorphism group of this surface will need to contain a subgroup isomorphic to $\mathbb{Z} \ast \mathbb{Z}$.

\par A related question is whether for smooth complex projective rational surfaces, there exists an upper bound on the number of isomorphism classes of real forms. We can answer this in the negative:
\begin{theorem} \label{MAIN THEOREM}
	For every positive integer $r$, there exists a smooth complex projective rational surface with at least $r$ real forms not pairwise isomorphic over $\mathbb{R}$.
\end{theorem}

\par Our proof of Theorem~\ref{MAIN THEOREM} hinges on the following construction: consider a real smooth cubic curve $C \subset \mathbb{P}^2_{\mathbb{C}}$ whose real locus $C(\mathbb{R})$ has two connected components in the Euclidean topology and for which the group of group automorphisms satisfies $\easy{Aut}_{\easy{gp}}(C) \cong \mathbb{Z}/2\mathbb{Z}$. If we fix an integer $r\geq3$ and choose points $p_{10}, \ldots, p_{r0} \in C(\mathbb{R})$ satisfying some suitable condition, then we can find exactly four points $p_{i1}, \ldots, p_{i4} \in C(\mathbb{R})$ for each $1 \leq i \leq r$ such that the tangent to $C$ at a point $p_{ij}$ with $j \neq 0$ cuts $C$ in $p_{i0}$. This situation allows us to define of cubic involutions $\sigma_1, \ldots, \sigma_r$ on the blow-up $X$ of $\mathbb{P}^2_{\mathbb{C}}$ in the $5r$ points $p_{ij}$. For each $i$, the automorphism $\sigma_i$ is the unique involution that fixes pointwise the strict transform of $C$ and preserves the strict transform of a general line through $p_{i0}$. These cubic involutions will be shown to induce $r$ nonisomorphic real forms on $X$, compare with Theorem~\ref{thm: main theorem rephrased explicitly}. The assumption that $r\geq 3$ is due to some computations in the proof of Theorem~\ref{thm: a=1, b=0 ... almost} (see Equations~\eqref{eq: for introduction, 1} and \eqref{eq: for introduction, 2}); the additional, tedious calculations for $r=2$ were left out but with some extra work, one could even show that for $r=2$, there are precisely $2$ real forms. In fact, we believe that all of the surfaces constructed in Theorem~\ref{thm: main theorem rephrased explicitly} have precisely $r$ forms, though we do not know how much more effort would be required to prove it.

\par After this paper appeared, Dinh, Oguiso and Yu \cite{dinh2021smooth} answered the Question~\ref{question} in the affirmative, and the author herself \cite{bot2021smooth} succeeded in constructing a smooth \textit{affine} rational surface defined over $\mathbb{C}$ with uncountably many nonisomorphic real forms. 

\par The structure of this article is the following: In Section~\ref{section: real structure}, we discuss the connection between real forms, real structures and automorphisms on $X$. We then introduce the cubic involutions in Section~\ref{section: Cubic involutions}, which will be the candidates for the inequivalent cocycles. We obtain desirable relations in the Picard group of $C$ in Section~\ref{section: independent points}, for which we need to assume certain points fulfilling an extra condition; with this condition, we can then start the translation of geometric properties of an automorphism of $X$ into arithmetic ones with the help of Section~\ref{section: condition on automorphisms of X}. As mentioned, along the way, we will have encountered a few conditions on the points chosen, and therefore, we secure their existence in Section~\ref{section: existence of suitable points}. In Sections~\ref{section: map from Pic(X) to Pic(C)}, \ref{section: description of induced automorphism in Pic(C)} and \ref{section: equations on the coefficients}, we will use the tools collected so far to complete the transition to arithmetic conditions on automorphisms of $X$. This allows us, finally, in Section~\ref{section: at least r real forms} to show Theorem~\ref{thm: main theorem rephrased explicitly}, which describes the construction for the smooth complex projective rational surface of Theorem~\ref{MAIN THEOREM}.
\par We fix $\mathbb{C}$ as the base field, and write $\mathbb{P}^2$ instead of $\mathbb{P}^2_{\mathbb{C}}$.
\par ~
\par \textbf{Acknowledgements.} I am very thankful to my PhD advisor J\'er\'emy Blanc for suggesting this topic and for the many discussions on it. Also, I would like to thank Fr\'ed\'eric Mangolte for reading and suggesting improvements for this text. Lastly, my thanks go to the referee for reading the paper very carefully and pointing out interesting details and subtleties.

\section{Real structures} \label{section: real structure}

\par The \textit{complexification} of a real algebraic variety $X_0$ is given by
	\[ (X_0)_{\mathbb{C}} \coloneqq X_0 \times_{\easy{Spec}\mathbb{R}} \easy{Spec}\mathbb{C}.\]
Fix a complex algebraic variety $X$. A \textit{real form} of $X$ is a real algebraic variety $X_0$ with a $\mathbb{C}$-isomorphism $(X_0)_{\mathbb{C}} \overset{\sim}{\rightarrow} X$. To understand the isomorphism classes of real forms of $X$, we instead study the real structures on $X$.
\par A \textit{real structure} on $X$ is an anti-regular involution $\rho : X \rightarrow X$, where the anti-regularity means that the following diagram commutes:
\begin{center}
	\begin{tikzcd}
		X \arrow[r, "\rho"] \arrow[d]  &  X \arrow[d] \\
		\easy{Spec}\mathbb{C} \arrow[r, "z \mapsto \overline{z}"]  & \easy{Spec}\mathbb{C} .
	\end{tikzcd}
\end{center}
Two real structures $\rho$ and $\rho'$ on the same complex variety $X$ are \textit{equivalent} if there exists $\varphi \in \easy{Aut}(X)$ such that $\rho=\varphi \rho' \varphi^{-1}$.

It is an exercise to show (see \cite[Prop.~1.1]{MR1954070} or \cite[Thm.~4.1]{benzerga}) that, if $n$ is even, there is, up to equivalence, only one real structure on $\mathbb{P}^n$, namely
\[[\, z_0 : \ldots : z_n \,] \mapsto [\, \overline{z_0} : \ldots : \overline{z_n} \,].\]
If, however, $n=2k+1$ is odd, there are two, up to equivalence:
\begin{align*}
	[\, z_0 : \ldots : z_n \,] &\mapsto [\, \overline{z_0} : \ldots : \overline{z_n} \,], \\
	[\, z_0 : z_1 : \ldots : z_n \,] &\mapsto [\, -\overline{z_1} : \overline{z_0}: \ldots : -\overline{z_{2k+1}} : \overline{z_{2k}} \,].
\end{align*}
Therefore, any real structure on $\mathbb{P}^2$ is equivalent to
\begin{equation} \label{eq: real structure on P^2}
 \widehat{\rho} : [\,z_0:z_1:z_2\,] \mapsto [\,\overline{z_0}:\overline{z_1}:\overline{z_2}\,].
\end{equation}

\par We can associate a real structure to a real form, and vice versa: considering a real form $X_0$ with complex isomorphism $\varphi: (X_0)_{\mathbb{C}}\overset{\sim}{\rightarrow}X$, we can set $\rho \coloneqq \varphi^{-1}\rho_0\varphi$ for the real structure $\rho_0 \coloneqq \easy{id} \times \easy{Spec}(z \mapsto \overline{z})$ on $(X_0)_{\mathbb{C}}$. Conversely, given a real structure $\rho$ on a complex variety $X$, the variety $X_0 \coloneqq X/\langle \rho \rangle$ is a real form of $X$.

\par As a matter of fact, there is an equivalence between the category of complex quasi-projective algebraic varieties with a real structure and the category of real quasi-projective algebraic varieties (see for example \cite{benzerga}, Chapter $3.1$, for more details).  Thus, knowing all real forms of a given complex variety is the same as knowing all real structures on it:

\begin{theorem}[{\cite[Thm.~$3.17$]{benzerga}}]\label{thm: real forms isomorphic if and only if real structures equivalent}
	Any two real forms of a complex quasi-projective variety $X$ are $\mathbb{R}$-isomorphic if and only if their associated real structures are equivalent.
\end{theorem}

\par With this theorem under our belt, we turn our attention to real structures. Fix a real structure $\rho$ on a complex variety $X$ and set $G \coloneqq \langle \rho \rangle$. Let $\rho'$ be another real structure on $X$. Then:
\[ \rho'  \rho \; \easy{\reflectbox{$\coloneqq$}} \; a_{\rho}\in \easy{Aut}(X).\]
Conversely, we can analyse the conditions on an automorphism of $X$ such that its composition with $\rho$ is again a real structure.

\begin{definition} \label{def: cocycle}
	An automorphism $a_{\rho}$ of $X$ is called a \textit{cocycle} if it verifies $(a_{\rho}\rho)^2 = \easy{id}_X$.
	Two cocycles $a_{\rho}, b_{\rho}$ are called \textit{equivalent} if there exists $\alpha \in \easy{Aut}(X)$ such that $ b_{\rho}\rho = \alpha^{-1}a_{\rho}\rho \alpha$.
	In that case, we write $a_{\rho} \sim b_{\rho}$.
	We denote by $Z^1(G, \easy{Aut}(X))$ the set of cocycles and by
	\[ H^1(G, \easy{Aut}(X)) \coloneqq Z^1(G, \easy{Aut}(X))/ \sim\]
	the first Galois cohomology set.
\end{definition}
Using $H^1(G, \easy{Aut}(X))$, we may describe the equivalence classes of real structures, and therefore of real forms: 

\begin{theorem}[{\cite[Section $2.6$]{MR181643}}] \label{thm: equivalence classes of real strucutres, bijection}
	Let $X$ be a quasi-projective complex variety and $\rho$ a real structure on $X$. The equivalence classes of real structures are in bijection with the elements of $H^1(\easy{Gal}(\mathbb{C}/\mathbb{R}), \easy{Aut}(X))$, where the nontrivial element of $\easy{Gal}(\mathbb{C}/\mathbb{R})$ acts on $\easy{Aut}(X)$ by conjugation with $\rho$.
\end{theorem}
In the context of a blow-up, we would like to preserve the real structure on a complex projective surface:

\begin{proposition}[{\cite[\RomanNumeralCaps{2}.6.1]{MR1015720}}] \label{prop: blow-up and real structure}
	Let $Y$ be a complex projective surface and $\widehat{\rho}$ a real structure on $Y$. The blow-up $\pi : X \rightarrow Y$ in a real point of $Y$ or in a pair of complex conjugated points of $Y$ allows one to give $X$ a real structure $\rho$ in a natural way such that $\pi$ is real, meaning $\pi \rho = \widehat{\rho} \pi$.
\end{proposition}
So if we blow-up $\mathbb{P}^2$ in real points or in pairs of conjugates, the real structure $\widehat{\rho}$ given in \eqref{eq: real structure on P^2} lifts to a real structure $\rho$ on the blow-up $X$. Therefore, by Theorem~\ref{thm: equivalence classes of real strucutres, bijection}, the equivalence classes of real structures on $X$ --- and hence the isomorphism classes of real forms of $X$ --- are in a one-to-one correspondence with the equivalence classes of cocycles of $X$ with respect to the real structure $\rho$.

\section{Cubic involutions} \label{section: Cubic involutions}

\par Theorem~\ref{thm: equivalence classes of real strucutres, bijection} replaces our search for a rational surface with $r$ real structures by finding one having $r$ inequivalent cocycles. In fact, we aim to construct a blow-up $X$ of $\mathbb{P}^2$ at real points having $r \geq 3$ automorphisms $\sigma_1, \ldots, \sigma_r$ such that
\begin{enumerate}[leftmargin=*]
	\item for all $1 \leq i \leq r$, the equation $(\sigma_i\rho)^2=\easy{id}_X$ holds, where $\rho$ is the standard real structure inherited from $\mathbb{P}^2$ (see Proposition~\ref{prop: blow-up and real structure}), \label{item: first item, cocycle}
	\item and, provided~\ref{item: first item, cocycle} holds, for $i \neq j$, the cocycles $\sigma_i$ and $\sigma_j$ are not equivalent, meaning that there does not exist any automorphism $\alpha \in \easy{Aut}(X)$ such that $\alpha\sigma_i\rho = \sigma_j \rho \alpha$.
\end{enumerate}

\par The construction we propose to achieve this relies on the following classical birational map (see \cite[Ex.\,$3$]{MR1282018}):
\begin{definition}
	Let $C \subset \mathbb{P}^2$ be a smooth cubic curve. For $p \in C$, denote by $\widehat{\sigma}_p : \mathbb{P}^2 \dashrightarrow \mathbb{P}^2$ the \textit{cubic involution centred at $p$}, which is the unique birational map defined in the following way: if $L$ is a general line containing $p$, then
	\begin{enumerate}[leftmargin=*]
		\item the map $\widehat{\sigma}_p$ satisfies $\widehat{\sigma}_p(L)=L$,
		\item the restriction $\widehat{\sigma}_p|_L$ is the unique involution that fixes the two points of $(L\cap C)\setminus \{p\}$ pointwise.
	\end{enumerate}
\end{definition}
Note that $\widehat{\sigma}_p$ restricts to the identity on the curve $C$.
In the sequel, we will always assume that $p$ in the above definition is not an inflection point. Then, by \cite[{Prop.\,$12$}]{MR2492397}, the base points of $\widehat{\sigma}_p$ are $p$ and those points $q$ where the tangent to $C$ at $q$ cuts $C$ in $p$; and there are precisely four such points. Denote them by $p_1, \ldots, p_4$ and call them \textit{associated with $p$}. 
\par Recall that for a fixed inflection point $p_0 \in C$ as neutral element we get a group structure on $C$. Furthermore, we have a group isomorphism
\begin{align*}
	C &\cong \easy{Pic}^0(C) \leq  \easy{Pic}(C), \\
	p &\mapsto p-p_0.
\end{align*}
In $\easy{Pic}(C)$, we write $p$ for the class of $p$. The condition that the tangent to the cubic curve at $p_i$ cuts it in $p$ can be expressed in $\easy{Pic}(C)$ by $p+2p_i=3p_0$. Since $p+2p_i=p+2p_j$ for $1 \leq i,j \leq 4$, we can deduce 
\[2(p_i-p_j) =0,\]
and hence $p_i-p_j \in C[2]$, where $C[2]$ denotes the group of $2$-torsion elements of $\easy{Pic}^0(C)$.
\par In for example \cite[Cor.~\RomanNumeralCaps{3}.6.4]{silverman2009arithmetic}, it is shown that $C[2] \cong ( \mathbb{Z} / 2 \mathbb{Z} )^2$, which we can use in the proof of the following lemma.

\begin{lemma} \label{lemma: 2 torsion and our points}
	Fix $p$ on a smooth cubic curve $C \subset \mathbb{P}^2$ which is not an inflection point, and call its associated points $p_1, \ldots, p_4$. Then, for every $1\leq i,j \leq 4$,
	\begin{align}
		p_{i}-p_{j}=p_{j}-p_{i}, \label{eq: permutation of indices}
	\end{align}
	and if $\{i,j,k,\ell \}=\{1,2,3,4\}$, then
	\[p_{i}-p_{j}=p_{k}-p_{\ell}.\]
	Moreover, $C[2]=\{ 0, p_{1}-p_{4},\, p_{2}-p_{4},\, p_{3}-p_{4} \}$.
\end{lemma}
\begin{proof}
	The first equality follows from $2(p_{i}-p_{j})=0$. To see that the second equality holds, note that we have --- thanks to the first equality --- six nonzero elements in total in $C[2]$:
	\[p_{1}-p_{2}, \, p_{1}-p_{3},\, p_{1}-p_{4},\, p_{2}-p_{3}, \, p_{2}-p_{4},\, p_{3}-p_{4}.\]
	But since $C[2] \cong (\mathbb{Z}/2\mathbb{Z})^2$ by \cite[Cor.~\RomanNumeralCaps{3}.6.4]{silverman2009arithmetic}, the $2$-torsion group $C[2]$ only contains three nonzero elements. If we equate two elements where an index agrees, then, possibly using~\eqref{eq: permutation of indices}, we find that two of the associated $p_{i}$ would have to agree. This is a contradiction to the points $p_{i}$ being distinct.
	\par As we have seen, the nontrivial elements of $C[2]$ can be given as $p_{1}-p_{4}$, $p_{2}-p_{4}$ and $p_{3}-p_{4}$, which concludes the proof.
\end{proof} 

\par We can now blow up the points $p$, $p_1, \ldots, p_4$ and examine what happens to $\widehat{\sigma}_p$.

\begin{proposition}\label{proposition: real points, cubic involutions commute with anti-regular involution}
	Let $C \subset \mathbb{P}^2$ be a real smooth cubic curve and suppose there exists a real point $p \in C$ which is not an inflection point. Consider the cubic involution $\widehat{\sigma}_p$ centred at $p$ and the rational surface $X$ obtained by blowing up $p, p_1, \ldots, p_4$. Then the birational map $\widehat{\sigma}_p$ lifts to an involution $\sigma_p \in \easy{Aut}(X)$ which commutes with the lift $\rho$ of the anti-regular involution $\widehat{\rho}: [\,x:y:z\,] \mapsto [\,\overline{x}:\overline{y}: \overline{z}\,]$ of the projective space. Furthermore, this implies $(\sigma_p\rho)^2=\easy{id}_X$, meaning $\sigma_p$ is a cocycle.
\end{proposition}
\begin{proof}
	Because $(\widehat{\sigma}_p)^2=\easy{id}_X$, and since $p$, $p_1, \ldots, p_4$ are the base points of $\widehat{\sigma}_p$, the birational map $\widehat{\sigma}_p$ lifts to an automorphism on the blow-up of the base points $p$, $p_1, \ldots, p_4$. Denote this lift by $\sigma_p \in \easy{Aut}(X)$. 
	\par By \cite[Ex.\,$3$]{MR1282018}, the birational map $\widehat{\sigma}_p$ is defined over $\mathbb{R}$ so long as $C$ is defined over $\mathbb{R}$ and $p$ is real. Therefore, on the blow-up in $p$, $p_1, \ldots, p_4$, the involutions $\sigma_p$ and $\rho$ commute.
\end{proof}

\par Suppose we can choose real points $p_{10}, \ldots, p_{r0}$ with $r\geq 3$ on a real smooth cubic curve $C$ such that none of them is an inflection point or associated points of one another. There will be further assumption on the points, which will be introduced in Section~\ref{section: independent points} and Section~\ref{section: condition on automorphisms of X} --- the existence of such points will be shown in Section~\ref{section: existence of suitable points}. Denote the associated points of $p_{i0}$ by $p_{i1}, \ldots, p_{i4}$. 
\par Under these assumptions, we can blow up all the $p_{ij}$'s to obtain a rational surface $X$, and for every $1\leq i \leq r$, the cubic involution $\widehat{\sigma}_{p_{i0}}$ lifts to an automorphism $\sigma_{p_{i0}} \; \reflectbox{$\coloneqq$}\; \sigma_i$. These automorphisms $\sigma_1, \ldots, \sigma_r \in \easy{Aut}(X)$ are the candidates for the inequivalent cocycles. Proposition~\ref{proposition: real points, cubic involutions commute with anti-regular involution} shows that they are indeed cocycles. 

\par The rest of the paper is devoted to proving that they are also inequivalent, which shows that there are at least $r$ real structures. In Section~\ref{section: condition on automorphisms of X}, we will express geometric properties of automorphisms of $X$ arithmetically, a translation needed to prove that the $\sigma_i$ are not equivalent.

\section{Independent points and relations in $\easy{Pic}(C)$} \label{section: independent points}

\par As mentioned in the previous chapter, we will discuss an important condition on the points $p_{10}, \ldots, p_{r0}$. For one, this condition will be indispensable for the translation of geometric properties into arithmetic ones, but it will also turn out to be stronger than the assumption of noncollinearity on the $p_{ij}$'s (see Corollary~\ref{cor: independent implies not collinear}):
\begin{definition} \label{def: independent}
	Let $C\subset \mathbb{P}^2$ a smooth cubic curve and $p_0 \in C$ an inflection point. Points $p_1, \ldots, p_s \in C$ are called \textit{independent} if in $\easy{Pic}^0(C)$, the elements $(p_1-p_0), \ldots, (p_s-p_0)$ are $\mathbb{Z}$-linearly independent.
\end{definition}
Note that independence implies that these points cannot be inflection points; in fact, $(p_i-p_0)$ is without torsion in $\easy{Pic}^0(C)$, or equivalently, $p_i$ is not torsion in $C$. We can show that the condition of independence on the $p_{10}, \ldots, p_{r0}$ is stronger than the condition of the existence four points out of all the $p_{ij}$'s, with no three collinear, for which we first show the following proposition:

\begin{proposition} \label{prop: all about relations on points}
	Let $r\geq 1$, fix independent points $p_{10}, \ldots, p_{r0}$ lying on a smooth cubic curve $C \subset \mathbb{P}^2$, choose an inflection point $p_0 \in C$ as neutral element and fix a labelling $\{0,\delta_1, \delta_2, \delta_3\}=C[2]\cong\left(\mathbb{Z}/2\mathbb{Z}\right)^2$. Then we can order the associated points $p_{i1}, \ldots, p_{i4}$ such that $\delta_k = p_{ik}- p_{i4}$ for $k=1,2,3$. Furthermore:
	\begin{enumerate}[leftmargin=*]
		\item We have the following relations in $\easy{Pic}(C)$:
		\begin{align*}
			3p_0 &= p_{10}+2p_{14},  \\
			p_{i0} &= p_{10}+2p_{14}-2p_{i4}, \\
			p_{i1} &= \delta_1+p_{i4},   \\
			p_{i2} &= \delta_2+p_{i4},  \\
			p_{i3} &= \delta_1+\delta_2+p_{i4},
		\end{align*}	
		for $1 \leq i \leq r$. \label{item: relations on points}
		\item If there exist $m, n_1, \ldots, n_r, s_1, s_2, d \in \Z$ such that 
		\begin{align*}
			mp_{10}+\underset{i=1}{\overset{r}{\sum}}n_ip_{i4}+s_1\delta_1+s_2\delta_2=3dp_0,
		\end{align*} 
		then $m=d$, $n_1=2d$, $n_2=\ldots=n_r=0$ and $s_1\equiv s_2\equiv 0 \, \easy{mod} 2$.
		\label{item: independence implies conditions}
	\end{enumerate}
\end{proposition}
\begin{proof}
	By Lemma~\ref{lemma: 2 torsion and our points}, we have $C[2]=\{p_{ik}-p_{i4} \thinspace | \thinspace k \in \{1,2,3,4\}\}$, so we may choose the order of the associated points such that $\delta_k=p_{ik}-p_{i4}$ for $k=1,2,3$. 
	\par As for \ref{item: relations on points}, the first equation is due to the nature of the points chosen, and the second equation follows from 
	\[p_{i0}+2p_{i4}=3p_0=p_{10}+2p_{14}.\]
	The last three equations hold by choice of the order of the associated points, where we use $\delta_3=\delta_1+\delta_2$. 
	\par To prove \ref{item: independence implies conditions}, we note that due to degree reason, we have $m+\overset{r}{\underset{i=1}{\sum}}n_i=3d$. Then, we multiply the equation by two and use $2\delta_1=2\delta_2=0$:
	\[2m(p_{10}-p_0)+\underset{i=1}{\overset{r}{\sum}}2n_i(p_{i4}-p_0)=0.\]
	Using $2(p_{i4}-p_0)=-(p_{i0}-p_0)$, we obtain
	\[(2m-n_1)(p_{10}-p_0)-\underset{i=2}{\overset{r}{\sum}}n_i(p_{i0}-p_0)=0.\]
	Since the $p_{i0}$'s are independent,
	\[2m=n_1, \quad n_2=\ldots=n_r=0,\]
	and thereby,
	\begin{align}
		mp_{10}+2mp_{14}+s_1\delta_1+s_2\delta_2=3dp_0. \label{eq: almost there}
	\end{align}
	As $\delta_1$ and $\delta_2$ are of degree $0$, we may deduce $3m=3d$, implying $m=d$ and $n_1=2d$. Due to $p_{10}+2p_{14}=3p_0$, we find from \eqref{eq: almost there} that
	\[s_1\delta_1+s_2\delta_2=0.\]
	Applying Lemma~\ref{lemma: 2 torsion and our points}, we conclude that $s_1\equiv s_2 \equiv 0\, \easy{mod} 2$. This completes the proof.
\end{proof}

With the above proposition, we can show that independence of the $p_{10}, \ldots, p_{r0}$ implies that no three points out of the $p_{ij}$'s are collinear.

\begin{corollary} \label{cor: independent implies not collinear}
	If the points $p_{10}, \ldots, p_{r0} \in C$ with $r\geq 1$ are independent, then no three points out of the $p_{ij}$'s, $1\leq i \leq r$, $0 \leq j \leq 4$, are collinear.
\end{corollary}
\begin{proof}
	Suppose by contradiction that $p_{ij}$, $p_{k\ell}$ and $p_{st}$ are collinear. By Proposition~\ref{prop: all about relations on points}, \ref{item: relations on points}, we know that we can write $p_{ij}+p_{k\ell}+p_{st}=3p_0$ as
	\begin{align}
		mp_{10}+\underset{i=1}{\overset{r}{\sum}}n_ip_{i4}+s_1\delta_1+s_2\delta_2=p_{ij}+p_{k\ell}+p_{st}=3p_0, \label{eq: rewriting p_ij+p_kl+p_st}
	\end{align}
	where $m, s_1, s_2 \in \{0,1,2,3\}$ and $n_1, \ldots, n_r \in \mathbb{Z}$.
	Using \ref{item: independence implies conditions}, we obtain $m=1$, $n_1=2$, $n_2= \ldots = n_r=0$ and $s_1\equiv s_2 \equiv 0\, \easy{mod} 2$. This implies, with \ref{item: relations on points}, that exactly one of the indices $j$, $\ell$, $t$ is equal to zero. Up to exchanging the points, we may assume $j=0$ and $\ell, t \neq 0$. Then, again with \ref{item: relations on points},
	\[p_{10}+2p_{14}=p_{i0}+p_{k\ell}+p_{st}=p_{10}+2p_{14}-2p_{i4}+p_{k\ell}+p_{st},\]
	which implies $2p_{i4}=p_{k\ell}+p_{st}$.
	Apply once more \ref{item: relations on points} to obtain
	\[2p_{i4}=p_{k\ell}+p_{st}=p_{k4}+p_{s4}+s_1'\delta_1+s_2'\delta_2,\]
	where $s_1', s_2' \in \{0,1,2\}$. With \ref{item: independence implies conditions}, this can only be if $k=s=i$ and $s_1'\equiv s_2' \equiv 0 \, \easy{mod} 2$. This implies $p_{k\ell}=p_{st}$, a contradiction to the points being distinct.
\end{proof}

As a consequence of Proposition~\ref{prop: all about relations on points}, we can find a specific base of the subgroup of $\easy{Pic}(C)$ generated by the $p_{ij}$'s.

\begin{proposition} \label{proposition: subgroup of Pic(C) with many points}
	Consider a smooth cubic curve $C\subset \mathbb{P}^2$ and an inflection point $p_0 \in C$ as neutral element. Fix a labelling $\{0,\delta_1, \delta_2, \delta_3\}=C[2]\cong\left(\mathbb{Z}/2\mathbb{Z}\right)^2$. Choose, for $r\geq 1$, independent points $p_{10}, \ldots, p_{r0} \in C$ and label their associated points $p_{ij}$ such that $ \delta_k =p_{ik} - p_{i4}$ for $k = 1, 2, 3$. Then we have an isomorphism
	\begin{align*}
		\raisebox{-1pt}{$\mathbb{Z}p_{10}$} \oplus \overset{r}{\underset{i=1}{\bigoplus}} \, \raisebox{-1pt}{$\mathbb{Z}p_{i4}$} \oplus  \overset{2}{\underset{j=1}{\bigoplus}} \, \raisebox{-1pt}{$( \mathbb{Z}/2\mathbb{Z})\delta_j$} & \,\raisebox{-1pt}{$\overset{\sim}{\longrightarrow} \langle \, p_{ij} \thinspace | \thinspace 1 \leq i \leq r, 0\leq j \leq 4 \, \rangle,$} \\
		(m,n_1, \ldots, n_r,s_1,s_2) &\mapsto  mp_{10} + \underset{i=1}{\overset{r}{\sum}} n_i p_{i4} + s_1\delta_1+s_2\delta_2.
	\end{align*}
\end{proposition}
\begin{proof}
	\par Surjectivity can be shown by hitting all the generators: this is guaranteed by the relations in \ref{item: relations on points} of Proposition~\ref{prop: all about relations on points}. For injectivity, we take $(m, n_1, \ldots, n_r, s_1, s_2)$ such that
	\[mp_{10} + \underset{i=1}{\overset{r}{\sum}} n_i p_{i4} + s_1\delta_1+s_2\delta_2=0.\]
	By Proposition~\ref{prop: all about relations on points}, \ref{item: independence implies conditions}, we find $m=n_1=\ldots=n_r=0$ and $s_1\equiv s_2 \equiv 0 \, \easy{mod} 2$, showing injectivity and finishing the proof.
\end{proof}

\section{Conditions on automorphisms of $X$} \label{section: condition on automorphisms of X}

Let $X$ be the blow-up of $\mathbb{P}^2$ in finitely many points $p_{ij}$. The action of $\easy{Aut}(X)$ on $\easy{Pic}(X)$ gives us a representation 
\begin{equation} \label{eq: representation}
	1 \rightarrow \easy{ker}\tau \hookrightarrow \easy{Aut}(X) \overset{\tau}{\rightarrow} \easy{Aut}(\easy{Pic}(X)).
\end{equation}
Any $g \in \easy{ker}\tau$ sends each $(-1)$-curve onto itself. Taking such a $g$, we find that it descends to an automorphism of $\mathbb{P}^2$ which fixes the points $p_{ij}$ pointwise. 
\par By Corollary~\ref{cor: independent implies not collinear}, this situation is granted for points $p_{ij}$ lying on a smooth cubic curve $C$, where $p_{10}, \ldots, p_{r0}$ are independent and $p_{ij}$ with $j \neq 0$ is associated with $p_{i0}$. Therefore, we can view $\easy{Aut}(X)$ as a subgroup of $\easy{Aut}(\easy{Pic}(X))$, and any $g \in \easy{Aut}(X)$ by its action $g^{\ast}$ on the Picard group
\[ \raisebox{-0.5pt}{$\easy{Pic}(X)=\mathbb{Z}[L]$} \oplus \underset{i,j}{\bigoplus} \,  \raisebox{-0.5pt}{$\mathbb{Z}[E_{ij}]$},\]
where $L$ is the strict transform of a line in $\mathbb{P}^2$ not passing through any of the $p_{ij}$'s, and the $E_{ij}$ are the exceptional curves of $X$.
\par What about $\sigma_1, \ldots, \sigma_r$ introduced in Section~\ref{section: Cubic involutions}? 
\begin{lemma}[{\cite[{Lemma\,$17$}]{MR2492397}}] \label{lemma: action of sigma_i}
	The induced action of $\sigma_i$ on $\easy{Pic}(X)$ is given by
	\begin{align*}
		\sigma_i^{\ast} : [L] &\mapsto 3[L]-2[E_{i0}]-[E_{i1}]- \cdots - [E_{i4}], \\
		[E_{i0}]&\mapsto 2[L]-[E_{i0}]-[E_{i1}]- \cdots - [E_{i4}], \\
		[E_{ij}] & \mapsto [L]-[E_{i0}] - [E_{ij}], \\
		[E_{k\ell}] & \mapsto [E_{k\ell}],
	\end{align*}
	for $j \neq 0$ and $k\neq i$.
\end{lemma}

\par Now that we know both $\easy{Aut}(X) \leq \easy{Aut}(\easy{Pic}(X))$ and what the action of the $\sigma_i$ looks like on $\easy{Aut}(\easy{Pic}(X))$, we can introduce helpful properties on automorphisms $g$ of $X$:
\begin{enumerate}[leftmargin=*]
	\item The induced automorphism $g^{\ast}$ of $\easy{Pic}(X)$ preserves the intersection form, meaning that for any two divisors $[D]$, $[D']$, we have 
	\begin{equation}
		g^{\ast}([D])\cdot g^{\ast}([D'])=[D]\cdot [D']. \tag{$\spadesuit$} \label{cond: first condition of stabiliser}
	\end{equation} 
	\item The canonical divisor is mapped to itself under $g^{\ast}$, meaning 
	\begin{equation}
		g^{\ast}(K_X)=K_X. \tag{$\diamondsuit$}\label{cond: second condition of stabiliser}
	\end{equation}
\end{enumerate}
With the help of the first condition, we can infer that automorphisms of $X$ commute with the anti-regular involution $\rho$, as long as $X$ is the blow-up of real points.

\begin{lemma} \label{lemma: real automorphisms}
	If $X$ is the blow-up in real points of $\mathbb{P}^2$, then all $(-1)$-curves are real. If, in addition, out of four of these points, no three are collinear, then all automorphisms are real.
\end{lemma}
\begin{proof}
	\par Recall that we denote by $\widehat{\rho}$ the real structure on $\mathbb{P}^2$ and by $\rho$ the unique real structure on $X$ such that the blow-up is real.
	\par Let $E$ be a $(-1)$-curve on $X$. Write it as
	\[E \sim dL-\underset{i}{\sum} m_i E_i,\]
	where $E_i$ are the exceptional curves of $X$ and $L$ is the pullback of a line of $\mathbb{P}^2$ not passing through the points blown-up. As we only blow up in real points, we have $\rho(E_i)=E_i$, since $\widehat{\rho}$ maps a real point to itself. We can thus calculate
	\[[\rho^{-1}(E)]=\rho^{\ast}([E])=d\rho^{\ast}([L])-\underset{i}{\sum} m_i \rho^{\ast}([E_i])=d[L]-\underset{i}{\sum} m_i [E_i]=[E],\]
	where we used that $\rho([L])=[L]$. Since $\rho(E)\cdot E=-1$, we get $\rho(E)=E$.
	\par For the second claim, consider $\alpha \in \easy{Aut}(X)$. For each $(-1)$-curve $E$, we have $\rho \alpha \rho(E)=\rho \alpha(E)$, since $\rho(E)=E$. Due to $\alpha$ preserving the intersection form by \eqref{cond: first condition of stabiliser}, the curve $\alpha(E)$ is also of self-intersection $-1$. As a result, $\rho \alpha(E)=\alpha(E)$.
	Hence, $\rho \alpha \rho(E)=\alpha(E)$. Consequently, $\alpha^{-1}\rho \alpha \rho$ maps every $(-1)$-curve to itself; as there are four points where no three are collinear, this implies that $\alpha^{-1}\rho \alpha \rho=\easy{id}_X$, as desired. 
\end{proof}
Thus, in our set-up, we have to assume that all associated points $p_{ij}$ are also real, so that for any automorphism $\alpha \in \easy{Aut}(X)$, we have $\alpha \rho = \rho \alpha$. This will then help in the proof of Theorem~\ref{thm: main theorem, at least r real forms}.

The second condition \eqref{cond: second condition of stabiliser} leads to the fact that any automorphism of $X$ must map the strict transform of the smooth cubic curve $C$ to itself:

\begin{lemma} \label{lemma: auto of X restricts to auto of C}
	Fix a smooth cubic curve $C \subset \mathbb{P}^2$, and choose at least ten points lying on $C$. Then any automorphism $g$ of the blow-up $X$ in those points restricts to an automorphism of the strict transform $\widetilde{C}$. 
\end{lemma}
\begin{proof}
	\par Every automorphism $g\in \easy{Aut}(X)$ induces an automorphism $g^{\ast}$ of $\easy{Pic}(X)$ which satisfies \eqref{cond: second condition of stabiliser}. Therefore, we find $[g^{-1}(\widetilde{C})]=g^{\ast}([\widetilde{C}])=g^{\ast}(-K_X)=-g^{\ast}(K_X)=-K_X=[\widetilde{C}]$,
	which implies that the two irreducible curves $\widetilde{C}$ and $g^{-1}(\widetilde{C})$ lie in the same divisor class. Furthermore, we can calculate $[\widetilde{C}]$ as $[\widetilde{C}]=3[L]-\sum[E_{i}]$,
	where $L$ is the pullback of a line in $\mathbb{P}^2$ not passing through the points blown up, and the $E_i$ are the exceptional curves.
	This implies that $[\widetilde{C}]^2=9-s$, where $s\geq 10$ is the number of points blown up. So we must have $[\widetilde{C}]^2<0$. From that we obtain $\widetilde{C} \cdot g^{-1}(\widetilde{C}) <0$, so $g^{-1}(\widetilde{C})=\widetilde{C}$. Since $\widetilde{C} \cong C$ is smooth, we find that $g|_{\widetilde{C}}$ is an automorphism.
\end{proof}
Note that since $C \cong \widetilde{C}$, we can view $g|_{\widetilde{C}}$ also as an automorphism of $C$.  In addition, blowing up at least ten points lying on a smooth cubic curve implies that $\tau$ in \eqref{eq: representation} is injective: any nontrivial automorphism of $X$ being mapped to the identity would fix at least ten exceptional curves and thus descend to an automorphism of $\mathbb{P}^2$ fixing at least ten points. Since any nontrivial automorphism of $\mathbb{P}^2$ having at least ten fixed points fixes four points lying on a line, this is a contradiction to B\'ezout.

\section{Existence of suitable points} \label{section: existence of suitable points}

\par Consider the following set-up resulting from the discussion in the previous sections: the $p_{i0}$'s are assumed to be independent as defined in Definition~\ref{def: independent}, any $p_{ij}$ with $j \neq 0$ is an associated point of $p_{i0}$ and they are all assumed to be real. As observed in Section~\ref{section: independent points}, independence excludes $3$-torsion elements, so none of the points can be an inflection point, and hence all have precisely four associated points. Therefore, we would like to prove the existence of points $p_{10}, \ldots, p_{r0}$ which are real, have only real associated points, and are independent. 
\par To achieve this, we will need to assume that the smooth cubic curve $C$ is defined over $\mathbb{R}$ and that the set of real points of $C$ has two components in the Euclidean topology. This assumption is equivalent to other useful statements:

\begin{lemma} \label{lemma: equivalent statements on real points of C}
	Let $C \subset \mathbb{P}^2$ be a smooth cubic curve defined over $\mathbb{R}$ and denote by $C(\mathbb{R})$ the real points of $C$, then the following three statements are equivalent:
	\begin{enumerate}[leftmargin=*]
		\item The set $C(\mathbb{R})$  has two connected components in the Euclidean topology. \label{item: connected components}
		\item For every group structure given to $C(\mathbb{C})$ by choosing a neutral point $p_0 \in C(\mathbb{R})$, the $2$-torsion elements are all real. \label{item: 2-torsion real}
		\item There exists a real point $p \in C(\mathbb{R})$ with four real associated points. \label{item: existence of good point}
	\end{enumerate}
	Moreover, if the above conditions are satisfied, then a real point has at least one real associated point if and only if it has only real associated points.
\end{lemma}
\begin{proof}
	We prove the implications in the following order: first, we prove the equivalence of \ref{item: connected components} and \ref{item: 2-torsion real}, and then that \ref{item: 2-torsion real} and \ref{item: existence of good point} are equivalent. 
	\par The equivalence of \ref{item: connected components} and \ref{item: 2-torsion real} is well-known; we nevertheless include its proof, as we could not find a reference. First, we make the following observation: choose $p_0 \in C$ a real inflection point and let the defining equation of $C$ in Weierstrass form be
	\[ C: Y^2Z=F(X,Z),\]
	where $F$ is homogeneous of degree three and has real coefficients. Dehomogenise this equation at $Z=1$, as the only point with $Z = 0$ is $p_0=[\,0:1:0\,]$ at infinity. Using the group law on $C$, the $2$-torsion points will then be precisely those points $[\, x : y : 1 \,]$ with $y=0$, meaning points $[\, x:0:1 \, ] \in C$ where $F(x,1)=0$.
	Since $C$ is smooth, this polynomial of degree three in $x$ has three distinct solutions $a_1$, $a_2$, $a_3$ in $\mathbb{C}$, and we can write
	\begin{align}
		y^2=F(x,1)=\lambda (x-a_1)(x-a_2)(x-a_3), \label{y^2=F(x,1)}
	\end{align}
	with $\lambda \in \mathbb{R}$. Replacing $x$ with $\tfrac{x}{\sqrt[3]{\lambda}}$, we may assume $\lambda=1$.
	\par We can rephrase the equivalence accordingly and instead prove that $C(\mathbb{R})$ has two connected components in the Euclidean topology if and only if $a_1$, $a_2$, $a_3$ are all real. As $F(x,1)=0$ has real coefficients, there are either three real solutions, or one real solution and two complex conjugate solutions. So we can show the equivalence by making a full classification. 
	\par First, if all solutions are real, up to permutation of the indices, we may assume $a_1 < a_2 < a_3$. Then $C(\mathbb{R}) \setminus \{[\,0:1:0\,]\}$ is precisely the set of points $[\, x:y:1\,]$ where $F(x,1)\geq0$, as then $y^2=F(x,1)$ has a real solution for $y$. Hence, we can describe the set $C(\mathbb{R})$ as
	\[C(\mathbb{R}) = \{\thinspace [\, x:\sqrt{F(x,1)}:1\,] \thinspace | \thinspace a_1 \leq x \leq a_2 \thinspace\} \cup \{\thinspace [\, x:\sqrt{F(x,1)}:1\, ] \thinspace | \thinspace a_3 \leq x < \infty \thinspace\} \cup \{[\, 0:1:0\, ]\}.\]
	The two components are therefore the first set and the union of the latter two sets.
	\par Assume now that, up to permutation of the indices, $a_1$ is real and $a_3=\overline{a_2}$. Since $(x-a_2)(x-\overline{a_2})$ is a polynomial with leading coefficient equal to one, we have $(x-a_2)(x-\overline{a_2})>0$ for all $x \in \mathbb{R}$. So, we can write $C(\mathbb{R})$ as the set
	\[C(\mathbb{R}) =  \{\thinspace [\, x:\sqrt{F(x,1)}:1\, ] \thinspace | \thinspace a_1 \leq x < \infty \thinspace\} \cup \{[\, 0:1:0\, ]\},\]
	which is comprised of one component.
	\par Now for the equivalence of \ref{item: 2-torsion real} and \ref{item: existence of good point}, call $p_1, \ldots, p_4$ the associated points of $p$. We recall that due to Lemma~\ref{lemma: 2 torsion and our points}, the group $C[2] \subset C(\mathbb{C})$ of (complex) $2$-torsion points is equal to 
	\[C[2]=\{ 0, \, p_1-p_4, \, p_2-p_4, \, p_3-p_4 \}.\]
	So, if the four associated points $p_1, \ldots, p_4$ of some real $p \in C$ are real, then the $2$-torsion group $C[2]$ contains only real points. This shows that \ref{item: existence of good point} implies \ref{item: 2-torsion real}.
	\par Suppose \ref{item: 2-torsion real} holds and consider the morphism $\varphi: C \rightarrow C$ sending a point $q$ to the second point on $C$ lying on the tangent to $C$ at $q$. Choose a point $p \in \varphi(C(\mathbb{R}))$ which is not an inflection point; this is possible by choosing $p$ as the image of a point other than an inflection point or a $2$-torsion point. So $p$ has at least one real associated point. Suppose without loss of generality that this point is $p_4$. Since $p_i - p_4 \in C[2]$ and $C[2] \subset C(\mathbb{R})$ by \ref{item: 2-torsion real}, the other associated points must also all be real. 
	\par As for the last assertion, if $C[2] \subset C(\mathbb{R})$ and $p$ has a real associated point, then by what just preceded, all its associated points must be real. This wraps up the proof.
\end{proof}

The next lemma introduces an uncountable set which will be crucial in the proof of the existence of suitable points.

\begin{lemma} \label{lemma: the set M}
	If $C[2] \subset C(\mathbb{R})$, then the set 
	\[ M \coloneqq \{ \thinspace p \in C(\mathbb{R}) \thinspace | \thinspace \forall q \in C: \bigl(p+2q=3p_0 \Rightarrow q \in C(\mathbb{R})\bigr) \thinspace \}\]
	is uncountable.
\end{lemma}
\begin{proof}
	Consider the morphism $\varphi: C \rightarrow C$ from the previous proof sending a point $q \in C$ to the second point on $C$ lying on the tangent to $C$ at $q$. Restricting it to $C(\mathbb{R})$, we obtain a continuous map $\overline{\varphi}: C(\mathbb{R}) \rightarrow C(\mathbb{R})$.
	\par The set $M$ is equal to $\overline{\varphi}(C(\mathbb{R}))$. Indeed, since all $2$-torsion elements are real, $p$ having any real associated point is equivalent to $p$ having only real associated points by Lemma~\ref{lemma: equivalent statements on real points of C}. The morphism $\varphi$ corresponds to multiplication by minus two when considering the group structure on $C$. Therefore, $\varphi$ has finite fibres, which implies the claim.
\end{proof}

Equipped with these lemmas, we can verify the existence of suitable points.

\begin{proposition} \label{proposition: existence of the good points}
	For any real smooth cubic curve $C \subset \mathbb{P}^2$ whose real points $C(\mathbb{R})$ has two components and for any $r \geq 1$ we can find points $p_{10}, \ldots, p_{r0} \in C$ such that for all $1\leq i \leq r$,
	\begin{enumerate}[leftmargin=*]
		\item the $p_{i0}$'s are independent, and \label{enum: condition special condition}
		\item the $p_{i0}$'s and their associated points are real. \label{enum: condition real points}
	\end{enumerate}
\end{proposition}
\begin{proof}
	We proceed by induction. Consider the uncountable set $M$ which describes the points of $C(\mathbb{R})$ having real associated points, as stated in Lemma~\ref{lemma: the set M}. 
	\par For $r=1$, independence is equivalent  to $p_{10}$ having no torsion. Since by, for example \cite[Cor.~\RomanNumeralCaps{3}.6.4]{silverman2009arithmetic}, the group of (complex) $m$-torsion points is $C[m] \cong \left( \mathbb{Z}/m\mathbb{Z} \right)^2$, the points of $m$-torsion are finite. Thus, the set of torsion points
	\[C_{\text{tors}} \coloneqq \underset{m\geq 1}{\bigcup} C[m]\]
	is countable. Since $M$ is uncountable, we can always choose a point $p_{10}$ from $M \setminus C_{\text{tors}}$ satisfying conditions \ref{enum: condition special condition} and \ref{enum: condition real points}. Note that this also excludes the inflection points, as they are precisely the $3$-torsion points.
	\par We assume that the proposition holds for $r-1 \geq 0$. Suppose we have points $p_{10}, \ldots, p_{r-1,0}$ satisfying \ref{enum: condition special condition} and  \ref{enum: condition real points}. The set
	\[ M_r \coloneqq M \setminus \{\thinspace p_{ij} \thinspace | \thinspace 1 \leq i \leq r-1, 1 \leq j \leq 4 \thinspace\}\]
	is still uncountable, as we only remove $5(r-1)$ points.
	\par Now, we would like to remove the points $q \in M_r$ for which there exist $(m, n_1, \ldots, n_{r-1}) \in \mathbb{Z}^r\setminus \{(0, \ldots, 0)\}$ such that
	\begin{align}
		m(q-p_0)+n_1(p_{10}-p_0)+\cdots +n_{r-1}(p_{r-1,0}-p_0)=0. \label{eq: bad condition}
	\end{align}
	Consider the morphism
	\[ \zeta_m : C \rightarrow C, \, q \mapsto mq, \]
	where $mq$ is the unique point corresponding to $m(q-p_0) \in \easy{Pic}^0(C)$. This morphism has finite fibres, which we can see by working on the complex torus analytically isomorphic to $C$ and using $C[m] \cong (\mathbb{Z}/m\mathbb{Z})^2$. Denote by 
	\[B_m \coloneqq \zeta_m^{-1}\{\, \text{point corresponding to }-n_1(p_{10}-p_0)-\cdots -n_{r-1}(p_{r-1,0}-p_0) \thinspace | \thinspace (n_1, \ldots, n_{r-1} ) \in \mathbb{Z}^r \,\},\]
	a countable subset of $C$. Therefore, 
	\[B \coloneqq \underset{m \in \mathbb{Z}}{\bigcup} B_m\]
	is also countable.
	\par Then $M_r \setminus B$ is still uncountably infinite, and we may choose $p_{r0} \in M_r \setminus B$. This proves the induction and therefore the proposition.
\end{proof}

\section{Map from $\easy{Pic}(X)$ to $\easy{Pic}(C)$} \label{section: map from Pic(X) to Pic(C)}

\par We would like to use the group structure on $C$ to find further conditions on automorphisms of $X$. For this, we take a step back from all the assumptions made so far: let $X$ be a smooth projective surface and let $C,D$ be two curves on $X$ having no common irreducible component. The \textit{intersection multiplicity} of $C$ and $D$ at a point $p \in X$ is defined as
\begin{equation*}
	i_p(C,D) \coloneqq \easy{dim}\left( \mathcal{O}_{X,p}/(f,g)\right),
\end{equation*} 
where $f$ and $g$ are local equations for $C$ and $D$ at $p$.

\par Fix the irreducible curve $C$ and consider the following map for irreducible curves $D \subset X$ which are not equal to $C$:
\begin{equation} \label{eqn: Phi map on irreducible curves}
	[D] \mapsto \underset{p \in C \cap D}{\sum} i_p(C, D)p.
\end{equation}
Considering divisors in $\easy{Pic}(X)$ as line bundles, we see that the above map extends to a unique $\mathbb{Z}$-linear map $\Phi: \easy{Pic}(X) \rightarrow \easy{Pic}(C)$.
\par Suppose now that $\pi: X \rightarrow \mathbb{P}^2$ is the blow-up in $5r \geq 10$ points $p_{10}, \ldots, p_{r4}$ lying on a smooth cubic curve $C$ chosen such that the $p_{i0}$'s are independent, and $p_{i1}, \ldots, p_{i4}$ are associated points of $p_{i0}$.
Then we can write the map as
\begin{align} \label{Phi}
	\Phi: \easy{Pic}(X) &\rightarrow \easy{Pic}(\widetilde{C}), \\ d[L]- \underset{i,j}{\sum} m_{ij} [E_{ij}] &\mapsto 3dp_0 - \underset{i,j}{\sum} m_{ij}p_{ij}, \nonumber
\end{align} 
where $\widetilde{C}$ is the strict transform of the cubic curve, $L$ is the strict transform of a general line in $\mathbb{P}^2$, and the $E_{ij} \coloneqq \pi^{-1}(p_{ij})$ are the exceptional curves. As observed in Lemma~\ref{lemma: auto of X restricts to auto of C}, any automorphism $g \in \easy{Aut}(X)$ restricts to an automorphism of $\widetilde{C}$, and we therefore obtain a commutative diagram
	\begin{center}
	\begin{tikzcd}
		\easy{Pic}(X) \arrow[r, "g^{\ast}"] \arrow[d, "\Phi"']  &  \easy{Pic}(X) \arrow[d, "\Phi"] \\
		\easy{Pic}(\widetilde{C}) \arrow[r, "(g|_{\widetilde{C}})^{\ast}"]  & \easy{Pic}(\widetilde{C}).
	\end{tikzcd}
\end{center}
Note that due to $\widetilde{C} \cong C$, we can replace $\easy{Pic}(\widetilde{C})$ by $\easy{Pic}(C)$.

\par Furthermore, we may connect $\Phi(\easy{Pic}(X))$ with what we found in Proposition~\ref{prop: all about relations on points}:

\begin{lemma} \label{lemma: Phi(Pic(X)) is equal to gp generated by points}
	Given $r\geq 2$ independent points $p_{10}, \ldots, p_{r0}$ lying on a smooth cubic curve $C\subset \mathbb{P}^2$, blow up these points and their associated points to obtain a rational surface $X$. Consider the map $\Phi$ given in \eqref{Phi}. Then $\Phi(\easy{Pic}(X))=\langle \, p_{10}, \ldots, p_{r4} \, \rangle$.
\end{lemma}
\begin{proof}
	Since $\Phi(\easy{Pic}(X))=\langle \, 3p_0, p_{10}, \ldots, p_{r4} \, \rangle$, we need to prove that $3p_0 \in \langle \, p_{10}, \ldots, p_{r4} \, \rangle$. This is due to $3p_0=p_{10}+2p_{14}$, as in Proposition~\ref{prop: all about relations on points}, \ref{item: relations on points}.
\end{proof}

\section{Description of induced automorphisms on $\easy{Pic}(C)$} \label{section: description of induced automorphism in Pic(C)}

\par Before reaping the benefits of the set-up of our points, it is worth analysing how an automorphism of a smooth cubic curve induces an automorphism of the Picard group of that curve.

\begin{lemma} \label{lemma: automorphism of C induces automorphism of Pic(C)}
	Fix a smooth cubic curve $C\subset \mathbb{P}^2$ with $\easy{Aut}_{\easy{gp}}(C) \cong \mathbb{Z}/2 \mathbb{Z}$, and choose an inflection point $p_0 \in C$. Any automorphism $h$ of $C$ may be given by 
	\begin{align*}
		h:C &\rightarrow C, \\
		p &\mapsto ap+b,
	\end{align*}
	with $a = \pm 1$ and $b \in C$. Then the induced automorphism $h^{\ast}$ of $\easy{Pic}(C)$ is given by
	\begin{align*}
		h^{\ast} : \easy{Pic}(C) & \rightarrow \easy{Pic}(C), \\
		D &\mapsto aD+\easy{deg}(D)B,
	\end{align*}
	where $B=b-p_0$ if $a=1$, or $B=b+p_0$ if $a=-1$, for a fixed inflection point $p_0$.
\end{lemma}
\begin{proof}
	The assumption on $\easy{Aut}_{\easy{gp}}(C)$ implies that every automorphism $h$ of $C$ is of the form $p \mapsto ap + b$ for some $a \in \{\pm 1\}$ and some $b \in C$. Such an automorphism $h$ induces an automorphism
	\begin{align*}
		h_{\ast}:\easy{Pic}^0(C) &\rightarrow \easy{Pic}^0(C), \\
		(p-p_0) &\mapsto a(p-p_0)+(b-p_0),
	\end{align*}
	\par We first consider $a=-1$ and $(b-p_0)=0$. Then $(p-p_0)$ is mapped to $h_{\ast}(p-p_0)=(h(p)-p_0)$ which satisfies
	\[(p-p_0)+(h(p)-p_0)=0.\]
	This implies, in $\easy{Pic}(C)$, that $h(p)=-p+2p_0$, and therefore that the automorphism $h:C \rightarrow C$, $p \mapsto -p$ induces the automorphism
	\begin{align*}
		h^{\ast}:\easy{Pic}(C) &\rightarrow \easy{Pic}(C), \\
		D &\mapsto -D+2\easy{deg}(D)p_0.
	\end{align*}
	\par We can analyse a pure translation, meaning $(p-p_0) \mapsto (p-p_0)+(b-p_0)$, in the same manner; we determine the point $h(p)$ for which $(h(p)-p_0)=h_{\ast}(p-p_0)=(p-p_0)+(b-p_0)$, which must be 
	\[h(p)=p+b-p_0.\]
	We therefore obtain the induced automorphism
	\begin{align*}
		h^{\ast}:\easy{Pic}(C) &\rightarrow \easy{Pic}(C), \\
		D &\mapsto D+\easy{deg}(D)(b-p_0).
	\end{align*}
	This covers $a=1$. We can obtain the case $a=-1$ by composing the inversion with a translation.
\end{proof}

Note that the condition $\easy{Aut}_{\easy{gp}}(C) \cong \mathbb{Z}/2\mathbb{Z}$ is satisfied for a general smooth cubic curve $C \subset \mathbb{P}^2$. 
\par The following lemma is a technical result which will be useful in the proof of the theorem thereafter. Its proof is dependent on condition \eqref{cond: first condition of stabiliser}, meaning on automorphisms preserving the intersection form.

\begin{lemma} \label{lemma: main idea of main proof}
	Given an automorphism $g \in \easy{Aut}(X)$, denote the corresponding matrix by $G$ and write $G_0$ for the first column and $G_{ij}$ for the column corresponding to the image of $[E_{ij}]$. For any row vector $v \coloneqq (a_0, a_{10}, \ldots, a_{r4}) \in \mathbb{Z}^{5r+1}$, we have
	\begin{align} \label{eq: special claim, 1}
		a_0[L]-\underset{i,j}{\sum} a_{ij}[E_{ij}]=g^{\ast}\bigl((v \cdot G_0)[L]-\underset{i,j}{\sum} (v \cdot G_{ij})[E_{ij}]\bigr),
	\end{align}
	and therefore,
	\begin{align} \label{eq: special claim, 2}
		a_0^2-\underset{i,j}{\sum} a_{ij}^2 =(v\cdot G_0)^2-\underset{i,j}{\sum} (v \cdot G_{ij})^2. 
	\end{align}
\end{lemma}
\begin{proof}
	Write
	\begin{equation} \label{eq: matrix Q}
		Q \coloneqq \left( \begin{array}{cccc} 
			1 & ~ &  ~ & ~ \\
			~ & -1 & ~ & ~ \\
			~ & ~ & \ddots & ~ \\
			~ & ~ & ~ & -1 
		\end{array} \right) 
	\end{equation}
	for the matrix associated to the intersection form on $\easy{Pic}(X)$, where
	\begin{align*}
		[L]^2&=1, &[E_{ij}]^2&=-1, &[L]\cdot [E_{ij}]&=0, &[E_{ij}]\cdot [E_{k\ell}]&=0 \text{ for } (i,j) \neq (k, \ell).
	\end{align*}
	We can express the fact that $g^{\ast}$ preserves the intersection form --- as stated in \eqref{cond: first condition of stabiliser} --- using $G$ and $Q$, namely,
	\[Q=GQG^T,\]
	where $G^T$ denotes the transpose of $G$.
	Therefore, we have
	\begin{align*}
		Qv^T &=GQG^Tv^T \\
		%&= GQ(vG)^T \\
		&=GQ(v\cdot G_0, v\cdot G_{10}, \ldots, v\cdot G_{r4})^T \\
		&=G(v\cdot G_0, -v\cdot G_{10}, \ldots, -v\cdot G_{r4})^T.
	\end{align*}
	This implies \eqref{eq: special claim, 1}. 	Equation \eqref{eq: special claim, 2} follows from \eqref{eq: special claim, 1} by taking the self-intersection, and using the fact that any automorphism preserves the intersection form.
\end{proof}

The next theorem is right at the intersection of the geometric and arithmetic world, and will prove crucial in securing inequivalency between the $\sigma_1, \ldots, \sigma_r$.

\begin{theorem} \label{thm: a=1, b=0 ... almost}
	Consider a smooth cubic curve $C \subset \mathbb{P}^2$ with $\easy{Aut}_{\easy{gp}}(C) \cong \mathbb{Z}/2\mathbb{Z}$ and let $r\geq 3$. Choose independent points $p_{10}, \ldots, p_{r0} \in C$ and blow them and their associated points up. Call the resulting rational surface $X$. Then, for each $g\in \easy{Aut}(X)$, the restricted map $g|_{\widetilde{C}}$ is a translation by a $2$-torsion element or the identity.  
\end{theorem}
\begin{proof}
	Let $g \in \easy{Aut}(X)$. Then $g$ induces an automorphism of $\easy{Pic}(X)$, which we may write as
	\begin{align*}
		g^{\ast}([L]) &= d[L]-\underset{i,j}{\sum} m_{ij}[E_{ij}], \\
		g^{\ast}([E_{k\ell}]) &= n_{k\ell}[L] - \underset{i,j}{\sum} e_{ij}^{k\ell}[E_{ij}],
	\end{align*}
	where $L$ is the strict transform of a general line not passing through any of the $p_{ij}$'s, and the $E_{ij}$ are the exceptional curves.
	\par As $r\geq 3$, we have $g(\widetilde{C})=\widetilde{C}$ by Lemma~\ref{lemma: auto of X restricts to auto of C}, where $\widetilde{C}$ is the strict transform of $C$ in $X$. Using the induced automorphism $(g|_{\widetilde{C}})^{\ast}$ on $\easy{Pic}(\widetilde{C})\cong \easy{Pic}(C)$, we find, thanks to Lemma~\ref{lemma: automorphism of C induces automorphism of Pic(C)}, elements $a = \pm 1$ and $B \in \easy{Pic}(C)$ such that
	\begin{align}
		3dp_0-\underset{i,j}{\sum} m_{ij}p_{ij}&= 3ap_0+3B, \label{eq: equations for p_0}   \\
		3n_{k\ell}p_0 - \underset{i,j}{\sum} e_{ij}^{k\ell}p_{ij}&=ap_{k\ell}+B. \label{eq: equations for p_ij}
	\end{align}
	The aim is to prove $a=1$ and $B \in C[2]$. Denote the matrix corresponding to $g^{\ast}$ by $G$, with columns $G_0$, $G_{ij}$ according to the basis $[L]$, $[E_{ij}]$.
	\par Recall that by Proposition~\ref{proposition: subgroup of Pic(C) with many points}, we have an isomorphism $\langle \, p_{10}, \ldots, p_{r4} \, \rangle \cong \mathbb{Z}^{r+1}\oplus \left( \mathbb{Z}/2\mathbb{Z}\right)^2$ with basis $p_{10}$, $p_{14}$, $p_{24}, \ldots, p_{r4}$, $\delta_1$, $\delta_2$; and furthermore, thanks to Lemma~\ref{lemma: Phi(Pic(X)) is equal to gp generated by points}, $\Phi(\easy{Pic}(X))=\langle \, p_{10}, \ldots, p_{r4} \, \rangle$.
	\par Writing \eqref{eq: equations for p_0} and \eqref{eq: equations for p_ij} in this basis gives us more information. We would like to write $B$ in this basis as well. This is possible since $(g|_{\widetilde{C}})^{\ast}$ sends $\Phi(\easy{Pic}(X))$ to $\Phi(\easy{Pic}(X))$, and therefore, for any $p \in \langle \, p_{10}, \ldots, p_{r4} \, \rangle$, we have $ap+B \in \langle \, p_{10}, \ldots, p_{r4} \, \rangle$, implying $B \in  \langle \, p_{10}, \ldots, p_{r4} \, \rangle$. Therefore, $B$ can be expressed as $B=(m,n_1, \ldots, n_r, s_1,s_2)$.
	\par We then use the relations in $\langle \, p_{10}, \ldots, p_{r4} \, \rangle$ of Proposition~\ref{prop: all about relations on points}, \ref{item: relations on points}, to write the left hand side of \eqref{eq: equations for p_ij} as
	\begin{align*}
	3n_{k\ell}p_0 - \underset{i,j}{\sum} e_{ij}^{k\ell}p_{ij} = & \; n_{k\ell}(p_{10}+2p_{14})- e_{10}^{k\ell}p_{10} - \underset{i \geq 2}{\sum}e_{i0}^{k\ell}(p_{10}+2p_{14}-2p_{i4})- \underset{i}{\sum}e_{i1}^{k\ell}(\delta_1+p_{i4}) \\
		& \;  -\underset{i}{\sum}e_{i2}^{k\ell}(\delta_2+p_{i4}) - \underset{i}{\sum}e_{i3}^{k\ell}(\delta_1+\delta_2+p_{i4}) - \underset{i}{\sum} e_{i4}^{k\ell}p_{i4}. 
	\end{align*}
	Next, we regroup and collect the terms to obtain
	\begin{align*}	
	3n_{k\ell}p_0 - \underset{i,j}{\sum} e_{ij}^{k\ell}p_{ij} =&\; \big(n_{k\ell}-\underset{i}{\sum} e_{i0}^{k\ell}\big)p_{10} +\big(2n_{k\ell}- \underset{i\geq 2}{\sum} 2e_{i0}^{k\ell}- \underset{j \neq 0}{\sum} e_{1j}^{k\ell}\big)p_{14} \\
		&\; + \big(2e_{20}^{k\ell}-\underset{j \neq 0}{\sum} e_{2j}^{k\ell}\big)p_{24}+ \cdots + \big(2e_{r0}^{k\ell}-\underset{j \neq 0}{\sum}  e_{rj}^{k\ell}\big)p_{r4} \\
		&\; + \big(-\underset{i}{\sum} e_{i1}^{k\ell}-\underset{i}{\sum} e_{i3}^{k\ell}\big)\delta_1 + \big(-\underset{i}{\sum} e_{i2}^{k\ell}-\underset{i}{\sum} e_{i3}^{k\ell}\big)\delta_2.
	\end{align*}
	Therefore,~\eqref{eq: equations for p_ij} can be rewritten as:
	\begin{align}  \label{eq: crucial eq for p_kl}
		\left( \begin{array}{c}
		n_{k\ell}-\underset{i}{\sum} e_{i0}^{k\ell} \\
		2n_{k\ell}- \underset{i\geq 2}{\sum} 2e_{i0}^{k\ell}- \underset{j \neq 0}{\sum} e_{1j}^{k\ell} \\
		2e_{20}^{k\ell}-\underset{j \neq 0}{\sum} e_{2j}^{k\ell} \\
		\vdots \\
		2e_{r0}^{k\ell}-\underset{j \neq 0}{\sum}  e_{rj}^{k\ell} \\
		-\underset{i}{\sum} \left( e_{i1}^{k\ell}+ e_{i3}^{k\ell}\right) \\
		-\underset{i}{\sum} \left( e_{i2}^{k\ell}+ e_{i3}^{k\ell}\right)
	\end{array}
	\right)=a\underline{\ell}_{k\ell} + \left(\begin{array}{c}
		\phantom{\underset{i}{\sum}} \hspace{-0.8em} m \\ \phantom{\underset{i}{\sum}} \hspace{-0.8em} n_1 \\ \phantom{\underset{i}{\sum}} \hspace{-0.8em} n_2 \\ \phantom{\underset{i}{\sum}} \hspace{-0.8em} \vdots \\ \phantom{\underset{i}{\sum}} \hspace{-0.8em} n_r \\ \phantom{\underset{i}{\sum}} \hspace{-0.8em} s_1 \\ \phantom{\underset{i}{\sum}} \hspace{-0.8em} s_2
	\end{array}\right),
	\end{align}
	where $\underline{\ell}_{k\ell}$ is the vector corresponding to $p_{k\ell}$ in the basis $p_{10}$, $p_{14}, \ldots, p_{r4}$, $\delta_1$, $\delta_2$. 
	\par Similarly, for \eqref{eq: equations for p_0}, 
	\begin{align} \label{eq: crucial eq for 3p_0}
		\left( \begin{array}{c}
		d-\underset{i}{\sum} m_{i0} \\
		2d-2 \underset{i \geq 2}{\sum} m_{i0}- \underset{j \neq 0}{\sum} m_{1j} \\
		2m_{20}-\underset{j \neq 0}{\sum} m_{2j} \\
		\vdots \\
		2m_{r0}-\underset{j \neq 0}{\sum}m_{rj} \\
		-\underset{i}{\sum} \left( m_{i1}+ m_{i3}\right) \\
		-\underset{i}{\sum} \left( m_{i2}+ m_{i3}\right)
	\end{array}
	\right)=a\underline{\ell}_{0} + 3\left(\begin{array}{c}
		\phantom{\underset{i}{\sum}} \hspace{-0.8em} m \\ \phantom{\underset{i}{\sum}} \hspace{-0.8em} n_1 \\ \phantom{\underset{i}{\sum}} \hspace{-0.8em} n_2 \\ \phantom{\underset{i}{\sum}} \hspace{-0.8em} \vdots \\ \phantom{\underset{i}{\sum}} \hspace{-0.8em} n_r \\ \phantom{\underset{i}{\sum}} \hspace{-0.8em} s_1 \\ \phantom{\underset{i}{\sum}} \hspace{-0.8em} s_2
	\end{array}\right)=\left(\begin{array}{c}
	\phantom{\underset{i}{\sum}} \hspace{-0.8em} a+3m \\
	\phantom{\underset{i}{\sum}} \hspace{-0.8em} 2a+3n_1 \\
	\phantom{\underset{i}{\sum}} \hspace{-0.8em} 3n_2 \\
	\phantom{\underset{i}{\sum}} \hspace{-0.8em} \vdots \\
	\phantom{\underset{i}{\sum}} \hspace{-0.8em} 3n_r \\
	\phantom{\underset{i}{\sum}} \hspace{-0.8em} 3s_1 \\
	\phantom{\underset{i}{\sum}} \hspace{-0.8em} 3s_2
\end{array} \right),
	\end{align}
	and here, $\underline{\ell}_0$ is the vector corresponding to $3p_0=p_{10}+2p_{14}$.
	\par We now use Lemma~\ref{lemma: main idea of main proof}: to show that $m=0$, consider the row vector $v=(v_0, v_{10}, \ldots, v_{r4})$ with entries in the basis $3p_0$, $p_{ij}$ for $1\leq i \leq r$, $0 \leq j \leq 4$ given by 
	\begin{align*}
		v_0 \coloneqq 1, \quad v_{ij} \coloneqq \begin{cases}
			1 & j=0, \\
			0 & j \neq 0.
		\end{cases}
	\end{align*}
	Then, applying~\eqref{eq: special claim, 2} of Lemma~\ref{lemma: main idea of main proof}, we obtain
	\[ 1-r=(v\cdot G_0)^2-\underset{k,\ell}{\sum} (v\cdot G_{k\ell})^2. \]
	Note that 
	\begin{align*}
		v\cdot G_0 &=d-\underset{i}{\sum} m_{i0} \overset{\eqref{eq: crucial eq for 3p_0}}{=}a+3m, \\
		v \cdot G_{k\ell} &= n_{k\ell}-\underset{i}{\sum} e_{i0}^{k\ell} \overset{\eqref{eq: crucial eq for p_kl}}{=} \begin{cases}
			m & \ell \neq 0, \\
			a+m & \ell = 0.
		\end{cases}
	\end{align*}
	Hence, we get
	\begin{align*}
		1-r&=(a+3m)^2-r(a+m)^2-4rm^2 \\
		&= (1-r)a^2+2(3-r)am+(9-5r)m^2.
	\end{align*}
	As $a^2=1$, we can deduce
	\begin{equation} \label{eq: for introduction, 1}
		m\bigl(2(3-r)a+(9-5r)m\bigr)=0,
	\end{equation}
	and therefore $m=0$ or $2(3-r)a+(9-5r)m=0$.
	\par If $r= 3$, then $2(3-r)a+(9-5r)m=0$ has the solution $m=0$, as desired. If $r \geq 4$, we note that $0<|m|= \tfrac{2(r-3)}{5r-9}<1$, so we cannot find another integer solution. Therefore, for any $r\geq 3$, we found $m=0$.
	\par As for $n_1$, we follow the same structure of argument as for $m$, but with the vector $w$ given by
	\begin{align*}
		w_0 \coloneqq 2, \quad w_{ij} \coloneqq \begin{cases}
			1 & i=1, j \neq 0, \\
			2 & i\geq 2, j=0, \\
			0 & \text{else.}
		\end{cases}
	\end{align*}
	Observing
	\begin{align*}
		w \cdot G_0 &= 2d-2 \underset{i \geq 2}{\sum} m_{i0}- \underset{j \neq 0}{\sum} m_{1j} \overset{\eqref{eq: crucial eq for 3p_0}}{=} 2a+3n_1, \\
		w \cdot G_{k\ell} &= 2n_{k\ell}- \underset{i\geq 2}{\sum} 2e_{i0}^{k\ell}- \underset{j \neq 0}{\sum} e_{1j}^{k\ell} \overset{\eqref{eq: crucial eq for p_kl}}{=} \begin{cases}
			  a+n_1 & k=1, \ell \neq 0, \\
			  2a+n_1 & k\geq 2, \ell = 0, \\
			  n_1 & \text{else},
		\end{cases}
	\end{align*}
	we can once more use Lemma~\ref{lemma: main idea of main proof} to find
	\begin{align*}
		4-4-4(r-1)&= (2a+3n_1)^2-4(a+n_1)^2-(r-1)(2a+n_1)^2-(4r-3)n_1^2 \\
						&=	-4a^2(r-1)-4(r-2)an_1+(9-5r)n_1^2.
	\end{align*} 
	Again using $a^2=1$, we see that
	\[n_1\bigl(-4(r-2)a+(9-5r)n_1\bigr)=0.\] 
	The second factor in the above equation cannot have an integer solution for $n_1$ if $r \geq 3$, as in this case, $0<|n_1|=\tfrac{4(r-2)}{5r-9}<1$. Therefore, $n_1=0$.
	\par Next, we analyse $n_s$ for $2 \leq s \leq r$. We choose $u_s$ to be the vector with entries
	\begin{align*}
		(u_s)_0 \coloneqq 0, \quad  (u_s)_{ij} \coloneqq \begin{cases}
			-2 & i=s, j=0,\\
			1 & i=s, j \neq 0,\\
			0 & \text{else}.
		\end{cases}
	\end{align*}
	Once again, we apply Lemma~\ref{lemma: main idea of main proof} using
	\begin{align*}
		u_s \cdot G_0 &= 2m_{s0}-\underset{j \neq 0}{\sum}m_{sj} \overset{\eqref{eq: crucial eq for 3p_0}}{=} 3n_s, \\
		u_s \cdot G_{k\ell} &= 2e_{s0}^{k\ell}-\underset{j \neq 0}{\sum} e_{sj}^{k\ell} \overset{\eqref{eq: crucial eq for p_kl}}{=} \begin{cases}
			-2a+n_s & k=s, \ell=0, \\
			a+n_s & k=s, \ell \neq 0\\
			n_s & \text{else}. 
		\end{cases}
	\end{align*}
	It implies
	\begin{align*}
		-8 &= 9n_s^2-5(r-1)n_s^2-4(a+n_s)^2-(-2a+n_s)^2 \\
		&= -8a^2-4an_s-(5r-9)n_s^2,
	\end{align*} 
	and with $a^2=1$, we get
	\begin{equation} \label{eq: for introduction, 2}n_s\bigl(4a+(5r-9)n_s\bigr)=0.\end{equation}
	One solution is $n_s=0$, so we need to determine the possible solutions of $4a+(5r-9)n_s=0$. Any such solution satisfies $0<|n_s|=\tfrac{4}{5r-9}<1$ if $r\geq 3$. Therefore, the only integer solution is $n_s=0$.
	\par Since $m=n_1=\ldots = n_r=0$, we find $B=s_1\delta_1 + s_2\delta_2 \in C[2]$. Furthermore, fixing $k=1$, $\ell=0$ and summing up the first $r+1$ entries of \eqref{eq: crucial eq for p_kl}, we find, as $\underline{\ell}_{10}=(1,0,\ldots, 0)^T$,
	\[3n_{10}-\sum e_{ij}^{10}=a+m+n_1 + \cdots + n_r = a.\]
	Thanks to Lemma~\ref{lemma: auto of X restricts to auto of C}, we get $g(\widetilde{C})=\widetilde{C}$; therefore, using that $g^{\ast}$ preserves the intersection form, we calculate
	\[3n_{10}-\sum e_{ij}^{10}=[\widetilde{C}]\cdot g^{\ast}([E_{10}])=g^{\ast}([\widetilde{C}])\cdot g^{\ast}([E_{10}])=[\widetilde{C}] \cdot [E_{10}]=1.\] 
	We deduce $a=1$, which finishes the proof.
\end{proof}

\section{Equations on the coefficients} \label{section: equations on the coefficients}

\par We saw in Theorem~\ref{thm: a=1, b=0 ... almost} that we can use the information given by the restriction of an automorphism of $X$ to an automorphism of the smooth cubic curve $C$. Using this, we can determine conditions on the coefficients:

\begin{proposition}\label{prop: coefficients if a=1,b=0}
	Consider the blow-up $X$ of $\mathbb{P}^2$ in the points $p_{ij}$ lying on a smooth cubic curve $C$, where $p_{10}, \ldots, p_{r0}$ with $r\geq 3$ are independent and if $j \neq 0$, then $p_{ij}$ is associated with $p_{i0}$. Let $g$ be an automorphism of $X$, given by
	\begin{align*}
		g^{\ast}([L]) &= d[L]-\underset{i,j}{\sum} m_{ij}[E_{ij}], \\
		g^{\ast}([E_{k\ell}])  &= n_{k\ell}[L]-\underset{i,j}{\sum} e_{ij}^{k\ell}[E_{ij}],
	\end{align*}
	with $E_{ij}$ the exceptional curves of $X$ and $L$ the strict transform of a line in $\mathbb{P}^2$ not passing through any of the points $p_{ij}$.
	Suppose that the restriction of $g$ to $\widetilde{C}$ is the identity or a translation by a $2$-torsion element. Then, we have the following conditions on the coefficients:
	\begin{align*}
		2e_{ij}^{i0}-1 &= e_{ij}^{i1} + \cdots + e_{ij}^{i4},  & 1 \leq i \leq r, \, j \neq 0.
	\end{align*}
\end{proposition}
\begin{proof}
	We work with the inverse $g^{-1}$, for which we know that $G^{-1}=QG^TQ$, with $Q$ as given in \eqref{eq: matrix Q} and $G$ the matrix corresponding to $g$. Therefore, we deduce that
	\begin{align*}
		(g^{-1})^{\ast}([L]) &=d[L]-\underset{k, \ell}{\sum} n_{k\ell}[E_{k\ell}], \\
		(g^{-1})^{\ast}([E_{ij}]) &= m_{ij}[L]- \underset{k, \ell}{\sum} e_{ij}^{k\ell}[E_{k\ell}].
	\end{align*}
	\par Since $(g^{-1})|_{\widetilde{C}}$ is a translation by a $2$-torsion element, we can write, as in the proof of Theorem~\ref{thm: a=1, b=0 ... almost},
		\[\left( \begin{array}{c}
		m_{ij}-\underset{k}{\sum} e_{ij}^{k0} \\
		2m_{ij}- \underset{k\geq 2}{\sum} 2e_{ij}^{k0}- \underset{\ell \neq 0}{\sum} e_{ij}^{1\ell} \\
		2e_{ij}^{20}-\underset{\ell \neq 0}{\sum} e_{ij}^{2\ell} \\
		\vdots \\
		2e_{ij}^{r0}-\underset{\ell \neq 0}{\sum}  e_{ij}^{r\ell} \\
		-\underset{k}{\sum} \left( e_{ij}^{k1}+ e_{ij}^{k3}\right) \\
		-\underset{k}{\sum} \left( e_{ij}^{k2}+ e_{ij}^{k3}\right)
	\end{array}
	\right)=\underline{\ell}_{ij} + \left(\begin{array}{c}
		\phantom{\underset{i}{\sum}} \hspace{-0.8em} 0 \\ \phantom{\underset{i}{\sum}} \hspace{-0.8em} 0 \\ \phantom{\underset{i}{\sum}} \hspace{-0.8em} 0 \\ \phantom{\underset{i}{\sum}} \hspace{-0.8em} \vdots \\ \phantom{\underset{i}{\sum}} \hspace{-0.8em} 0 \\ \phantom{\underset{i}{\sum}} \hspace{-0.8em} s_1 \\ \phantom{\underset{i}{\sum}} \hspace{-0.8em} s_2
	\end{array}\right),\]
	where $\underline{\ell}_{ij}$ is the vector corresponding to $p_{ij}$ in the basis $p_{10}$, $p_{14}, \ldots, p_{r4}$, $\delta_1$, $\delta_2$, and $B=s_1\delta_1+s_2\delta_2$. 
	Consider $i \geq 2$ and $j \neq 0$. Then $\underline{\ell}_{ij}=e_{i+1}+f_{ij}$, where $f_{ij}$ is a vector with zero entries except maybe in the last two spots, meaning $f_{ij}$ corresponds to an element in $C[2]$. From that, we deduce
	\[2e_{ij}^{i0}= e_{ij}^{i1} + \cdots + e_{ij}^{i4}+1.\]
	For $i=1$ and $j \neq 0$, we analyse the first two entries together. Since $\underline{\ell}_{1j}=e_2+f_{1j}$, with again $f_{1j}$ a vector corresponding to an element in $C[2]$, we see that
	\begin{align}
		m_{1j}-\underset{k}{\sum} e_{1j}^{k0} &= 0, \label{eq: first eq for eqs on coeff} \\
		2m_{1j}- \underset{k\geq 2}{\sum} 2e_{1j}^{k0}- \underset{\ell \neq 0}{\sum} e_{1j}^{1\ell} &= 1. \label{eq: second eq for eqs on coeff}
	\end{align}
	We can subtract the double of \eqref{eq: first eq for eqs on coeff} from \eqref{eq: second eq for eqs on coeff}, which implies the desired equation, and hence concludes the proof.
\end{proof}

\section{At least $r$ real forms}\label{section: at least r real forms}

We are finally able to prove that given our setting, the $\sigma_i$ are inequivalent, which is the last step needed to prove Theorem~\ref{MAIN THEOREM}:

\begin{theorem} \label{thm: main theorem, at least r real forms}
	Fix $r\geq 3$ and a real smooth cubic curve $C\subset \mathbb{P}^2$ whose set of real points $C(\mathbb{R})$ has two connected components in the Euclidean topology, and which satisfies $\easy{Aut}_{\easy{gp}}(C) \cong \mathbb{Z}/2\mathbb{Z}$. Consider points $p_{ij} \in C$, $1\leq i \leq r$, $0 \leq j \leq 4$, such that the points are all real, the $p_{ij}$ for a fixed $i$ and $1 \leq j \leq 4$ are associated to $p_{i0}$, and the points $p_{10}, \ldots, p_{r0}$ are independent. The blow-up $X$ of $\mathbb{P}^2$ in the points $p_{ij}$ has automorphisms $\sigma_1, \ldots, \sigma_r$ arising from the cubic involutions centred at $p_{10}, \ldots, p_{r0}$. If $i \neq j$, then the anti-regular automorphisms $\sigma_i\rho$ and $\sigma_j\rho$ are not equivalent, meaning there does not exist an automorphism $\alpha \in \easy{Aut}(X)$ such that $\sigma_i \rho \alpha  = \alpha\sigma_j \rho$.
\end{theorem}
\begin{proof}
	The existence of points $p_{ij}$ which are real is granted by Proposition~\ref{proposition: existence of the good points} and by $C(\mathbb{R})$ having two connected components in the Euclidean topology. Then, thanks to Lemma~\ref{lemma: real automorphisms} and all the points being real, we have $\alpha \rho=\rho \alpha$ for any automorphism $\alpha \in \easy{Aut}(X)$. Therefore, we can show instead that there does not exist $\alpha \in \easy{Aut}(X)$ such that $\sigma_i \alpha  = \alpha\sigma_j$.
	\par We may assume without loss of generality that $i=1$, $j=2$. Suppose by contradiction that there exists some automorphism $\alpha$ such that $\sigma_1 \alpha = \alpha\sigma_2$, where $\alpha$ is given by
	\begin{align*}
		\alpha^{\ast}([L]) &= d[L]-\underset{i,j}{\sum} m_{ij}[E_{ij}], \\
		\alpha^{\ast}([E_{k\ell}]) &= n_{k\ell}[L]-\underset{i,j}{\sum} e_{ij}^{k\ell}[E_{k\ell}],
	\end{align*}
	 where the $E_{ij}$ are the exceptional curves and $L$ is the pullback of a line in $\mathbb{P}^2$ not passing through any of the $p_{ij}$'s. We consider the image of $[L]$ under $\sigma_1^{\ast} \alpha^{\ast}$ and $\alpha^{\ast} \sigma_2^{\ast}$ and compare coefficients; in fact, the coefficient of $[E_{21}]$ will be enough. In these calculations, we will use both the action of $\alpha$ on $\easy{Pic}(X)$ and Lemma~\ref{lemma: action of sigma_i}. Now, on the one hand:
	 \begin{align*}
	 	\sigma_1^{\ast} \alpha^{\ast} ([L]) &= \sigma_1^{\ast}(d[L]-\underset{i,j}{\sum} m_{ij}[E_{ij}]) \\
	 	&= -m_{21}[E_{21}] + \text{ other terms.}
	 \end{align*}
 	On the other hand:
 	\begin{align*}
 		\alpha^{\ast} 	\sigma_2^{\ast} ([L]) &= \alpha^{\ast}(3[L]-2[E_{20}]-[E_{21}]-\cdots-[E_{24}]) \\
 		&= -(3m_{21}-2e^{20}_{21}-e^{21}_{21}-\cdots -e^{24}_{21})[E_{21}] + \text{ other terms.}
 	\end{align*} 
We can now equate the coefficients
\[-(3m_{21}-2e^{20}_{21}-e^{21}_{21}-\cdots -e^{24}_{21})=-m_{21}.\]
To this equation we can, thanks to $\easy{Aut}_{\easy{gp}}(C)\cong \mathbb{Z}/2\mathbb{Z}$ and in turn Theorem~\ref{thm: a=1, b=0 ... almost}, apply Proposition~\ref{prop: coefficients if a=1,b=0} and obtain
\[ 2m_{21}-4e^{20}_{21}+1=0. \]
But this is a contradiction modulo $2$, as the coefficients can only take integer values. This completes the proof.
\end{proof}

The following theorem ascertains the existence of a smooth projective rational surface as stated in Theorem~\ref{MAIN THEOREM}.

\begin{theorem} \label{thm: main theorem rephrased explicitly}
	Consider a real smooth cubic curve $C \subset \mathbb{P}^2$ with $\easy{Aut}_{\easy{gp}}(C)\cong \mathbb{Z}/2\mathbb{Z}$ and whose set of real points $C(\mathbb{R})$ has two connected components in the Euclidean topology. Fix $r\geq 3$. Let $p_{ij}$, $1\leq i \leq r$, $0 \leq j \leq 4$, be points on $C$ satisfying the conditions:
	\begin{enumerate}[leftmargin=*]
		\item The $p_{10}, \ldots, p_{r0}$ are independent, \label{cond: independence}
		\item for a fixed $1 \leq i \leq r$, the points $p_{i1}, \ldots, p_{i4}$ are associated with $p_{i0}$, \label{cond: associated}
		\item and all the $p_{ij}$ are real. \label{cond: realness}
	\end{enumerate}
	Call $X$ the blow-up of $\mathbb{P}^2$ in the points $p_{ij}$. Then $X$ has at least $r$ real forms, namely the ones corresponding to the inequivalent cocycles $\sigma_1, \ldots, \sigma_r$, the lifts of cubic involutions centred at $p_{10}, \ldots, p_{r0}$.
\end{theorem}
\begin{proof}
	\par The existence of points satisfying \ref{cond: independence}, \ref{cond: associated} and \ref{cond: realness} is given by Proposition~\ref{proposition: existence of the good points}. We can therefore consider the rational surface $X$ --- the blow-up of $\mathbb{P}^2$ in the points $p_{ij}$ --- and automorphisms $\sigma_1, \ldots, \sigma_r$ on $X$, which are cocycles by Proposition~\ref{proposition: real points, cubic involutions commute with anti-regular involution}.
	\par These automorphisms $\sigma_1, \ldots, \sigma_r$ are pairwise inequivalent by Theorem~\ref{thm: main theorem, at least r real forms}. The first cohomology set $H^1(\easy{Gal}(\mathbb{C}/\mathbb{R}), \easy{Aut}(X))$ therefore contains at least the equivalence classes of these $r$ cocycles. From this, we find by Theorem~\ref{thm: equivalence classes of real strucutres, bijection} that the rational surface $X$ we constructed has at least $r$ real structures, and therefore $r$ real forms, by Theorem~\ref{thm: real forms isomorphic if and only if real structures equivalent}, concluding the proof.
\end{proof}

\small{

\bibliographystyle{alpha}
%\bibliography{refs}

\begin{thebibliography}{BCTSSD85}
	
	\bibitem[Ben16a]{MR3473660}
	M.~Benzerga,
	\newblock {\em Real structures on rational surfaces and automorphisms acting
	trivially on {P}icard groups},
	\newblock {Math. Z.} {\bf 282(3-4)} (2016), 1127--1136.
	
	\bibitem[Ben16b]{benzerga}
	M.~Benzerga,
	\newblock {Structures réelles sur les surfaces rationnelles},
	\newblock PhD thesis, Université d'Angers, Angers, 2016.
	
	\bibitem[Bla08]{MR2492397}
	J.~Blanc,
	\newblock {\em On the inertia group of elliptic curves in the {C}remona group of the
	plane},
	\newblock {Michigan Math. J.} {\bf 56(2)} (2008), 315--330.
	
	\bibitem[BS64]{MR181643}
	A.~Borel and J.-P. Serre,
	\newblock {\em Th\'{e}or\`emes de finitude en cohomologie galoisienne},
	\newblock {Comment. Math. Helv.} {\bf 39} (1964), 111--164.
	
	\bibitem[Bot21]{bot2021smooth}
	A.~Bot,
	\newblock {\em A smooth complex rational affine surface with uncountably many real forms},
	\newblock {arXiv e-prints}, page arXiv:2105.08044, July 2021.
	
	\bibitem[CF19]{MR4023392}
	A. Cattaneo and L. Fu,
	\newblock {\em Finiteness of {K}lein actions and real structures on compact
	hyperk\"{a}hler manifolds},
	\newblock {Math. Ann.} {\bf 375(3-4)} (2019), 1783--1822.

	\bibitem[DFMJ20]{dubouloz2020smooth}
	A.~Dubouloz, G.~Freudenburg, and L.~Moser-Jauslin,
	\newblock {\em Smooth rational affine varieties with infinitely many real forms},
	\newblock {Journal f{\"u}r die reine und angewandte Mathematik} {\bf 771} (2021), 215--226.
	
	\bibitem[DIK00]{MR1795406}
	A.~Degtyarev, I.~Itenberg, and V.~Kharlamov,
	\newblock {\em Real {E}nriques surfaces}, volume 1746 of {``Lecture Notes in Mathematics},
	\newblock Springer-Verlag, Berlin, 2000.
	
	\bibitem[DO19]{MR3934593}
	T.-C.~Dinh and K.~Oguiso,
	\newblock {\em A surface with discrete and nonfinitely generated automorphism group},
	\newblock {Duke Math. J.} {\bf 168(6)} (2019), 941--966.
	
	\bibitem[DOY20]{2020arXiv200204737D}
	T.-C. {Dinh}, K. {Oguiso}, and X. {Yu},
	\newblock {\em Projective rational manifolds with non-finitely generated discrete automorphism group and infinitely many real forms},
	\newblock {arXiv e-prints}, page arXiv:2002.04737, February 2020.
	
	\bibitem[DOY21]{dinh2021smooth}
	T.-C. {Dinh}, K. {Oguiso}, and X. {Yu},
	\newblock {\em Smooth complex projective rational surfaces with infinitely many real forms},
	\newblock {arXiv e-prints}, page arXiv:2106.05687, June 2021.
	
	\bibitem[Giz94]{MR1282018}
	M.~H.~Gizatullin,
	\newblock {\em The decomposition, inertia and ramification groups in birational geometry},
	\newblock In {``Algebraic geometry and its applications'' ({Y}aroslavl',
		1992)}, Aspects Math., E25, pages 39--45. Friedr. Vieweg, Braunschweig, 1994.
	
	\bibitem[Kha02]{MR1936747}
	V.~Kharlamov,
	\newblock {\em Topology, moduli and automorphisms of real algebraic surfaces},
	\newblock {Milan J. Math.} {\bf 70} (2002), 25--37.
	
	\bibitem[Kim20]{MR4060185}
	J.~H. Kim,
	\newblock {\em On the finiteness of real structures of projective manifolds},
	\newblock {Bull. Korean Math. Soc.} {\bf 57(1)} (2020), 109--115.
	
	\bibitem[{Les}18]{MR3773792}
	J.~Lesieutre,
	\newblock {\em A projective variety with discrete, non-finitely generated
	automorphism group},
	\newblock {Invent. Math.} {\bf 212(1)} (2018), 189--211.
	
	\bibitem[Rus02]{MR1954070}
	F.~Russo,
	\newblock {\em The antibirational involutions of the plane and the classification of
	real del {P}ezzo surfaces},
	\newblock {Algebraic geometry} (2002), 289--312. de Gruyter, Berlin.
	
	\bibitem[Sil82]{MR678890}
	R.~Silhol,
	\newblock {\em Real abelian varieties and the theory of {C}omessatti},
	\newblock {Math. Z.} {\bf 181(3)} (1982), 345--364.
	
	\bibitem[Sil89]{MR1015720}
	R.~Silhol,
	\newblock {\em Real algebraic surfaces}, volume 1392 of {``Lecture Notes in
		Mathematics''},
	\newblock Springer-Verlag, Berlin, 1989.
	
	\bibitem[Sil09]{silverman2009arithmetic}
	J.~H.~Silverman,
	\newblock {``The arithmetic of elliptic curves''}, volume 106,
	\newblock Springer Science \& Business Media, 2009.

	
	
\end{thebibliography}

}
\noindent

\bigskip\noindent
Anna Bot ({\tt annakatharina.bot@unibas.ch}),\\
Department of Mathematics and Computer Science, University of Basel, 4051 Basel, Switzerland

\end{document}